\documentclass[sn-mathphys,Numbered]{sn-jnl} 

\usepackage{graphicx}%
\usepackage{multirow}%
\usepackage{amsmath,amssymb,amsfonts}%
\usepackage{amsthm}%
\usepackage{mathrsfs}%
\usepackage[title]{appendix}%
\usepackage{xcolor}%
\usepackage{textcomp}%
\usepackage{manyfoot}%
\usepackage{booktabs}%
\usepackage{algorithm}%
\usepackage{algorithmicx}%
\usepackage{algpseudocode}%
\usepackage{listings}%
\usepackage{enumerate}
 
\theoremstyle{thmstyleone}%
  \newtheorem{theorem}{Theorem}[section]
\newtheorem{lemma}[theorem]{Lemma}

\newtheorem{hypothesis}[theorem]{Hypothesis} 
 
\theoremstyle{thmstyletwo}%

\theoremstyle{thmstylethree}%
\newtheorem*{remark}{Remark}%

 \numberwithin {equation} {section} 

\begin{document}

\title[block-transitive $3$-$(v,k,1)$ designs]{Block-transitive $3$-$(v,k,1)$ designs on exceptional groups of Lie type}

\author[1]{\fnm{Ting} \sur{Lan}}\email{lanting0603@163.com}
\equalcont{These authors contributed equally to this work.}

\author[2,1]{\fnm{Weijun} \sur{Liu}}\email{wjliu6210@126.com}
\equalcont{These authors contributed equally to this work.}

\author*[1]{\fnm{Fu-Gang} \sur{Yin}}\email{18118010@bjtu.edu.cn}
\equalcont{These authors contributed equally to this work.}

\affil[1]{\orgdiv{School of Mathematics and Statistics}, \orgname{Central South University}, \orgaddress{\street{Yuelu South Road}, \city{Changsha}, \postcode{410083}, \state{Hunan}, \country{P.R. China}}}
 
\affil[2]{\orgdiv{College of General Education}, \orgname{Guangdong University of Science and Technology}, \orgaddress{\city{Dongguan}, \postcode{523083}, \state{Guangdong}, \country{P.R. China}}}

\abstract{Let $\mathcal{D}$ be a  non-trivial $G$-block-transitive $3$-$(v,k,1)$ design, where $T\leq G \leq \mathrm{Aut}(T)$ for some finite non-abelian simple group $T$. 
It is proved that if $T$  is a  simple exceptional group of Lie type, then $T$ is either the  Suzuki group ${}^2B_2(q)$ or $G_2(q)$.
Furthermore, if  $T={}^2B_2(q)$ then the design $\mathcal{D}$ has parameters $v=q^2+1$ and $k=q+1$, and so  $\mathcal{D}$ is an inverse plane of order $q$; and if  $T=G_2(q)$  then the point stabilizer in $T$ is either $\mathrm{SL}_3(q).2$ or $\mathrm{SU}_3(q).2$, and the  parameter $k$ satisfies very restricted conditions. }


 
\maketitle

 \section{Introduction}\label{sec:1}

 For an integer $t\geq 2$, a $t$-$(v,k,1)$ design  $\mathcal{D}= (\mathcal{P},\mathcal{B})$  is a set $\mathcal{P}$ of $v$  \emph{points} and a collection $\mathcal{B}$ of  $k$-subsets of $\mathcal{P}$,  called \emph{blocks},  such that, each block in $\mathcal{B}$ has size $k$, and each $t$-subset of $\mathcal{P}$ lies in a unique block.
If $t<k<v$, then we say $\mathcal{D}$  \emph{non-trivial}. 
 The $t$-$(v,k,1)$ designs are also known as \emph{Steiner $t$-designs}.

An automorphism of $\mathcal{D}$ is a permutation on $\mathcal{P}$ which permutes the blocks among themselves.
For a subgroup $G$ of the automorphism group $\mathrm{Aut}(\mathcal{D})$ of $\mathcal{D}$, the design $\mathcal{D}$ is said to be \emph{$G$-block-transitive} if $G$ acts transitively on the set of blocks, and is said to be \emph{block-transitive} if it is $\mathrm{Aut}(\mathcal{D})$-block-transitive. The point-transitivity and flag-transitivity are defined similarly (a flag of $\mathcal{D}$ is a  pair $(\alpha,B)$ where $\alpha $ is a point and $B$ is a block containing $\alpha$).

Flag-transitive  non-trivial $t$-$(v,k,1)$ designs has been classifed by~Buekenhout et. al~\cite{BDDKLS1990} for $t=2$ (apart from the one-dimensional affine group case), 
and by Cameron and Praeger~\cite{CP1993lagert} for $t\geq 7$ (they proved more generally that there exists no non-trivial $t$-designs with $t\geq 8$), and by Huber~\cite{H2005,H2007-4flag,H2009-5flag,H2010t=7,H2010} for  $3\leq t \leq 7$.
Notice that flag-transitivity implies  block-transitivity, while the converse is not true.

A celebrate theorem of Cameron and Praeger~\cite[Theorem~2.1]{CP1993lagert} (based on result of Ray-Chaudhuri and Wilson~\cite{RCW1975}) asserts that the block-transitive automorphism group of a  $t$-design is a $\lfloor t/2 \rfloor$-homogeneous permutation group on the point set.
Let $\mathcal{D}$ be a non-trivial $G$-block-transitive $t$-$(v,k,1)$ design.
If $t\geq 4$, then $G$ is a well known $2$-homogeneous group (see~\cite{K1972}) and is  affine or almost simple.
Huber proved in~\cite{H2010t=7} that $t\neq 7$, and in~\cite{H2010} that if $t=6$ then  $G=\mathrm{P\Gamma L}_2(p^e)$ with $p\in \{2,3\}$ and $e$ an odd prime power. 
Huber also proved in~ \cite{H2010affine} that, if $t \in \{4,5\}$ and $G$ is affine, then $t=4$ and $G$ is a one-dimensional affine group.  
 
For the cases $t=2$ and $3$, $G$ is $1$-homogeneous, which means that $G$ can be any transitive permutation group. 
This makes the classification of block-transitive $t$-$(v,k,1)$ design with $t\in \{2,3\}$ much open.
There have been a great deal of efforts to classify block-transitive $2$-$(v,k,1)$ designs. 
In 2001, Camina and Praeger~\cite{CP2001} proved that
if $t=2$ and $G$ is quasiprimitive on the point set, then $G$ is either affine or  almost simple.
The case where $G$ is almost simple has been widely studied, such as alternating groups~\cite{CNP2003},  sporadic simple groups~\cite{CS2000}, simple groups of Lie type of  small ranks~\cite{G2007, L2001, L2003, L2003a, L2003b,LLG2006, LLM2001, LZLF2004,   Z2002, Z2005,ZLL2000}, and large dimensional classical groups~\cite{CGZ2008}.  

In this paper, we focus on the case $t=3$. 
For $t=3$, Mann and Tuan~\cite[Corollary~2.3(a)]{MT2001} have shown that  $G$ is primitive on the point set (based on the result of Cameron and Praeger~\cite{CP1993} on block-transitive and point-imprimitive $t$-designs). 
Recently, it is proved by Gan and the second author~\cite{GL2022+} that  $G$ is either  affine or almost simple.
Inspired by this reduction result, we study non-trivial $G$-block-transitive $3$-$(v,k,1)$ design with $G$ an almost simple group.  
The case where the socle of $G$ is an alternating group has been settled by the authors~\cite{LLY2022+}. 
Hence here, we continue to consider  simple  exceptional groups of Lie type, which consists of ten infinite families of simple groups:
${}^2B_2(q)$, ${}^2G_2(q)$, ${}^2F_4(q)$, ${}^3D_4(q)$, ${}^2E_6(q)$, $G_2(q)$, $F_4(q)$, $E_6(q)$, $E_7(q)$, $E_8(q)$. 

Before stating our result, we recall some background of the block-transitive $3$-$(v,k,1)$ design on the Suzuki group ${}^2B_2(q)$. For a positive integer $n$, a $3$-$(n^2+1,n+1,1)$ design is also called a (finite) \emph{inversive plane} (following Dembowski~\cite[p.103]{Dembowski-book}). 
It is proved by L\"{u}neburg~\cite[Theorem~1]{L1964} that if $q=2^e$ with odd $e>1$, then there is one, and up to isomorphism only one, inversive plane of order $q$ which admits an automorphism group isomorphic to  ${}^2B_2(q)$. 
By~\cite[p.275,4]{Dembowski-book}, the  full automorphism group of this inversive plane of order $q$ is  $\mathrm{Aut}({}^2B_2(q))\cong {}^2B_2(q)\,{:}\, e$.
 This inversive plane of order $q$ is ${}^2B_2(q)$-block-transitive by~\cite[Corollary~1]{L1964}, while it is not  flag-transitive. 

\begin{theorem}\label{th:exceptional}  
Let $\mathcal{D}$ be a  non-trivial $G$-block-transitive $3$-$(v,k,1)$ design, where $T\leq G \leq \mathrm{Aut}(T)$ and $T$ is a simple exceptional group of Lie type.  
Then $T$ is either ${}^2B_2(q)$ or $G_2(q)$.  
\begin{enumerate}[\rm (a)]
\item If $T={}^2B_2(q)$, where $q=2^e$ with odd $e>1$, then  $G=T.f$ with $f\mid e$, $v=q^2+1$ and $k=q+1$ and so $\mathcal{D}$ is an inversive plane of order $q$.
\item If $T=G_2(q)$, where $q=p^e$ with prime $p$, then the point stabilizer in $T$ is $ \mathrm{SU}_3(q).2$ or $\mathrm{SL}_3(q).2$, and moreover, $q>10^5$.  
\end{enumerate}  
\end{theorem}
We remark that, in part (b) of Theorem~\ref{th:exceptional}, the parameter $k$  satisfies very limited  conditions.  
It shall be proved that if the point stabilizer in $T$ is $ \mathrm{SU}_3(q).2$, then  $v=q^3(q^3-1)/2$ and $k$ satisfies the following conditions:
\begin{itemize}
\item $k<\sqrt{v}+2$,
\item $k-2 \mid v-2$, $(k-1)(k-2)\mid (v-1)(v-2)$ and $k(k-1)(v-2) \mid v(v-1)(v-2)$,
\item $4(v-1)(v-2)/(q^3+1)=(q^3-2)(q^6-q^3-4) \mid 192k(k-1)(k-2)e$,
\end{itemize} 
and if the point stabilizer in $T$ is $ \mathrm{SL}_3(q).2$, then $v=q^3(q^3+1)/2$ and $k$ satisfies the first and second conditions above, together with
\begin{itemize} 
\item $4(v-1)(v-2)/(q^3-1)=(q^3+2)(q^6+q^3-4) \mid 192k(k-1)(k-2)e$.
\end{itemize} 
Despite a lot of efforts, we still can't prove the non-existence of such design.

In the process of proving Theorem~\ref{th:exceptional}, we find  some interesting results,  which we believe would be useful for studying block-transitive $t$-$(v,k,1)$ designs.

Following~\cite{KMMM2013}, a permutation group on a finite set is said to be \emph{quasi-semiregular} if it fixes only one point and acts semiregularly on the remaining points.  
We find that if a $t$-$(v,k,1)$ design with $t\geq 3$ admits a quasi-semiregular automorphism group of order $t-1$ or $t$, then $v-1$ is divisible by $k-1$ or $k$ (see Lemma~\label{lm:order3} for detail). 
In the work of Huber~\cite{H2005} on classifying $G$-flag-transitive $3$-$(v,k,1)$ designs, 
he spent about $5$ pages to deal with the case where ${}^2G_2(q)\leq G\leq \mathrm{Aut}({}^2G_2(q))$ with its $2$-transitive action.  
Our finding helps us give  a conciser proof for the ${}^2G_2(q)$ case (see Lemma~\ref{lm:2G2q}).  

Let $\mathcal{D}=(\mathcal{P},\mathcal{B}) $ be a non-trivial  $G$-block-transitive $3$-$(v,k,1)$ design, and let $\alpha \in \mathcal{P}$ and $B\in \mathcal{B}$.
It is proved in Lemma~\ref{lm:akbv}  that  
\[ \vert G_{\alpha} \vert (  \vert G_{\alpha} \vert / \vert G_{B} \vert  +2)^2> \vert G \vert .\]
Following Alavi and Burness~\cite{AB2015}, a proper subgroup $Y$ of a finite group $X$ is  said to be \emph{large} if $ \vert Y \vert ^3\geq  \vert X \vert $.
The above inequality implies that $G_\alpha$ is almost large in $G$.
We further prove that if $G$ is almost simple, then $ G_\alpha$ must be large in $G$. 
Large maximal subgroups of  almost simple  exceptional groups  of Lie type is determined  by Alavi, Bayat and Daneshkhah~\cite{ABD2022}.  
Based on their results, we obtain all candidates for $(G,G_\alpha)$ satisfying the assumption of Theorem~\ref{th:exceptional}.
It is also proved in Lemma~\ref{lm:akbv} that
\[\mathrm{gcd}( \vert G_\alpha \vert ,(v-1)(v-2))>\sqrt{v}-2.\]
Since $ \vert G_\alpha \vert = \vert G \vert /v$, the above inequality is only on the parameter $v$.
This  inequality  turns out to be very useful, because it has helped us rule out almost all candidates for $(G,G_\alpha)$ in proving Theorem~\ref{th:exceptional}.

 \section{Preliminaries}
For notations of groups, we refer to the Atlas~\cite{Atlas}. 
In particular, for two finite groups $H$ and $K$, we  use $H.K$, $H\,{:}\, K$ and $H \times K$ to present a extension of $H$ by $K$, a split extension (or semidirect product) of  $H$ by $K$, and the direct product of $H$ and $K$, respectively. 
For a positive integer $n$, we also use $n$ to denote the cyclic group of order $n$, and use $[n]$ for a solvable group of order $n$ of unspecified structure. 
The \emph{socle} of a finite group is the product of all its minimal normal subgroups.
A group $G$ is said to be \emph{affine} or \emph{almost simple} if its socle is an elementary abelian group or a non-abelian simple group.
Notice that if $G$ is almost simple, then $T\leq G \leq \mathrm{Aut}(T)$ for some non-abelian simple group $T$.

Our symbols and terms for permutation groups follow Dixon and Mortimer's book~\cite{DM-book}. 
Let $G$ be a finite permutation group on a finite set $\Omega$. 
For a point $\alpha \in \Omega$, denote by $G_\alpha$ the stabilizer of $\alpha$ in $G$, and by $\alpha^G$ the orbit of $G$ containing $\alpha$.

For two positive integers $m$ and $n$, let $\mathrm{gcd}(m,n)$   be their  largest common divisor.  
For a prime power $q\geq 2$ and  an integer $n\geq 3$, a prime $r$ is called a {\em primitive prime divisor} of $q^n-1$ if $r\mid (q^n-1)$ but $r\nmid (q^i-1)$ for any positive integer $i<n$. Primitive prime divisor was first discovered by Zsigmondy~\cite{Zsigmondy}, and the following lemma can be found in~\cite[Lemma~1.13.3]{BHRD2013}.

\begin{lemma}\label{lm:ppd}
Let $q\geq 2$ be a prime power and $n\geq 3$ with $(q,n)\neq (2, 6)$. Then $q^n-1$ has at least one primitive prime divisor $r$. Furthermore, $r\equiv 1\pmod{n}$.
\end{lemma}

 The next result shall be used in the proof of Theorem~\ref{th:exceptional}.
\begin{lemma}\label{le:divide}
Let $a,b$ be positive integers and let $r$ be a prime. Then $r^a-1$ divides $r^b-1$ if and only if $a$ divides $b$.
\end{lemma}

\begin{proof}
If  $a$ divides $b$, that is, $b=ma$ for some positive integer $m$,   then 
\[
(r^b-1)=(r^{am}-1)=(r^a-1)(r^{a(m-1)}+\cdots +r^a+1),\]
and hence $r^a-1$ divides $r^b-1$.
Next we assume that $r^a-1$ divides $r^b-1$.
Suppose for a contradiction that $a$ does not divide $b$. 
Let $b=xa+c$, where $1\leq c\leq a-1 $.
Since  $r^a-1$ divides both $r^b-1$ and $r^{xa}-1$, 
we see that $r^a-1$ divides $(r^b-1)-(r^{xa}-1)=r^b-r^{xa}=r^{xa}(r^c-1)$.
Notice that $\gcd(r^{xa},r^a-1)=1$, since $r^{xa}-(r^a-1)(r^{a(x-1)}+\cdots +r^a+1)=r^{xa}-(r^{xa}-1)=1$. 
It follows that $r^a-1$ divides $r^c-1$,   contradicting $c\leq a-1$. 
\end{proof}

For a general $3$-$(v,k,1)$ design, the parameters $v$ and $k$ have the next important relation, which can be obtained from~\cite[Theorem 3A.4]{C1976}. 

\begin{lemma}[{\cite{C1976}}]\label{lm:vk}
Let $\mathcal{D} $ be a non-trivial $3$-$(v,k,1)$ design. Then $v-2\geq (k-1)(k-2)$. 
In particular, $k< \sqrt{v} +2$.
\end{lemma}

The parameters $v$ and $k$ satisfy the following  well-known divisibility conditions (see~\cite[p.180]{DM-book} for example).
 \begin{lemma}\label{lm:brlambda2}
Let $\mathcal{D}=(\mathcal{P},\mathcal{B})$ be a $3$-$(v,k,1)$ design.  
\begin{enumerate}[\rm (a)]
\item $\displaystyle  \vert \mathcal{B} \vert =\frac{v(v-1)(v-2)}{k(k-1)(k-2)}$. 
 

\item For every point $\alpha$ in $\mathcal{P}$, the number of blocks containing $\alpha$ is 
\[
\lambda_1:=\frac{(v-1)(v-2)}{(k-1)(k-2)}. 
\]

  \item For two distinct points  $\alpha$ and $\beta$ in $\mathcal{P}$, the number of blocks containing  $\alpha$ and $\beta$ is 
 \[
\lambda_2:=\frac{ v-2 }{ k-2 }. 
\]

\end{enumerate} 
\end{lemma}

In the rest of this section, we  collect and prove  some conclusions on block-transitive $3$-$(v,k,1)$ design. 
Notice that a $G$-block-transitive $t$-$(v,k,1)$ design  is $G$-point-transitive (see~\cite[Theorem~6.2B]{DM-book} for example).
Furthermore, due to Mann and Tuan~\cite[Corollary~2.3(a)]{MT2001}, $G$ acts primitively on the point set. 
\begin{lemma}[{\cite{MT2001}}]\label{lm:VP}
Let $\mathcal{D}$ be a  $G$-block-transitive $3$-$(v,k,1)$ design. Then $G$ acts primitively on the point set of $\mathcal{D}$.
\end{lemma}
 
 The following lemma  plays a key role in our proof. 
 
\begin{lemma}\label{lm:akbv}
Let $\mathcal{D}=(\mathcal{P},\mathcal{B}) $ be a  $G$-block-transitive $3$-$(v,k,1)$ design.  
Let $\alpha \in \mathcal{P}$ and $B\in \mathcal{B}$.
 \begin{enumerate}[\rm (a)] 
\item  $ \vert G_\alpha \vert k(k-1)(k-2)= \vert G_B \vert (v-1)(v-2)$. 
\item  $ \vert G_{\alpha} \vert (  \vert G_{\alpha} \vert / \vert G_{B} \vert  +2)^2> \vert G \vert $.  
\item  $ \mathrm{gcd}( \vert G_\alpha \vert ,(v-1)(v-2))>\sqrt{v}-2$.
\end{enumerate} 
\end{lemma}

\begin{proof}
 Notice that the block-transitivity implies the point-transitivity.
Therefore, $v \vert G_\alpha \vert = \vert G \vert = \vert \mathcal{B} \vert \cdot  \vert G_{B} \vert $. 
By Lemma~\ref{lm:brlambda2}(a),  we have
\[
v \vert G_\alpha \vert =\frac{v(v-1)(v-2)}{k(k-1)(k-2)}  \vert G_{B} \vert ,
\]  
which leads to part (a) of the lemma.
 
 By Lemma~\ref{lm:vk} and  part (a), we have
\begin{equation}\label{eq:abv2}
\frac{ \vert G_{\alpha} \vert }{ \vert G_{B} \vert }=\frac{(v-1)(v-2)}{k(k-1)(k-2)}\geq \frac{v-1}{k}>\frac{v-4}{\sqrt{v} +2}=\sqrt{v}-2.
\end{equation} 
Since $v= \vert G \vert / \vert G_\alpha \vert $ by the point-transitivity of $G$, it follows 
$(G_{\alpha} \vert / \vert G_{B} \vert +2)^2> \vert G \vert / \vert G_\alpha \vert $  and so $ \vert G_{\alpha} \vert (  \vert G_{\alpha} \vert / \vert G_{B} \vert  +2)^2> \vert G \vert $, proving part (b).
  
From part (a)  we see that  $ \vert G_{\alpha} \vert $ divides $ \vert G_{B} \vert (v-1)(v-2)$, which implies that $ \vert G_{\alpha} \vert $ divides $ \vert G_{B} \vert  \cdot\mathrm{gcd}( \vert G_{\alpha}  \vert ,(v-1)(v-2))$.
Consequently, 
\[ \vert G_{\alpha}  \vert / \vert G_B \vert \leq \mathrm{gcd}( \vert G_{\alpha}  \vert ,(v-1)(v-2)).\]
This together with $ \vert G_{\alpha} \vert / \vert G_{B} \vert >\sqrt{v}-2$ lead  to part (c). 
\end{proof}


Recall that a permutation group is said to be \emph{semiregular} if the stabilizer of every point is trivial, and that a permutation group is said to be \emph{quasi-semiregular} if it fixes only one point and acts semiregularly on the remaining points.
For a $G$-block-transitive $3$-$(v,k,1)$ design, it is pointed out by~\cite[Lemma~2.5]{GL2022+} that if $G$ contains a semiregular automorphism group of order $3$, then $k$ divides $v$.
Inspired by their idea, we obtain more similar results.

\begin{lemma}\label{lm:order3}
Let $\mathcal{D}=(\mathcal{P},\mathcal{B}) $ be a $G$-block-transitive $t$-$(v,k,1)$ design with $t\geq 3$, and let $H$ be a subgroup of $G$. 
\begin{enumerate}[\rm (a)]
\item If $H$  has order $t$ and is semiregular on $\mathcal{P}$, then $v$  is divisible by $ k$. 
\item If $H$ has order $t$ and  is quasi-semiregular on $\mathcal{P}$, then $v-1 $ is divisible by  $k$ or $k-1$. 
\item If $H$ has order $t-1$ and  is quasi-semiregular on $\mathcal{P}$, then $v-1 $ is divisible by   $k-1$.
\end{enumerate} 
\end{lemma}
\begin{proof}

(a) The case $t=3$ is proved in~\cite[Lemma~2.5]{GL2022+}, and the general case can be proved similarly. 
For the reader's convenience, we prove it here.

Let $O_1,\ldots,O_m $ be all orbits of $H$ on $\mathcal{P}$. 
Since $H$ is semiregular, we see $ \vert O_i \vert =t$ for every $i\in \{1,\ldots,m\}$.
Let $B_1,...,B_s$ be all blocks containing one of $\{O_1,\ldots,O_m \}$.
Then $\mathcal{P} =\bigcup_{i=1}^{m}O_i \subseteq \bigcup_{j=1}^{s}B_j$.
By the definition of $t$-$(v,k,1)$ design, $O_i$  is contained in exactly one block for every $i\in \{1,\ldots,m\}$, which implies that $B_j$ is fixed by $H$ and hence $B_j$ is the union of some orbits of $H$ for every $j\in \{1,\ldots,s\}$.
Moreover, for two distinct $j_1,j_2 \in \{1,\ldots,s\}$, the intersection of the blocks $B_{j_1}$ and $B_{j_1}$ is the empty set.
Consequently, $\{B_1,...,B_s \}$ is a partition of  $\{O_1,\ldots,O_m \}$ and so $\mathcal{P} =\bigcup_{i=1}^{m}O_i =\bigcup_{j=1}^{s}B_j$.
It follows that $k$ divides $v$.

(b) Since $H$ is quasi-semiregular, $H$ fixes a unique point, namely  $\alpha$. 
Let  $O_0,O_1,\ldots,O_m $ be all orbits of $H$ on $\mathcal{P}$, where $O_0=\{ \alpha \} $. 
Then $ \vert O_0 \vert =1$ and $ \vert O_i \vert =t$ for every $i\in \{1,\ldots,m\}$.  
Let $B_1,...,B_s$ be all blocks containing one of $\{O_1,\ldots,O_m \}$.
Then  $B_j$ is fixed by $H$ and hence $B_j$ is the union of some orbits of $H$ for every $j\in \{1,\ldots,s\}$.
Now there are two cases according to whether $\alpha$ is in $B_j$.
If $\alpha \notin B_j$ for any  $j\in \{1,\ldots,s\}$, then $\mathcal{P}\setminus \{ \alpha \}=\bigcup_{i=1}^{m}O_i =\bigcup_{j=1}^{s} B_j $, and hence $v-1$ is divisible by $k$.

Assume that $\alpha \in B_j$ for some  $j\in \{1,\ldots,s\}$. Then $k= \vert B_j \vert \equiv 1\pmod{t}$.
For every $\ell \in \{1,\ldots,s\}$, since $B_\ell$ is the union of some orbits of $H$  we conclude that $\alpha \in B_\ell$.
Then  $\mathcal{P}\setminus \{ \alpha \}=\bigcup_{i=1}^{m}O_i =\bigcup_{\ell=1}^{s}(B_\ell\setminus \{ \alpha\})$, and hence $v-1$ is divisible by $k-1$.

(c) Now we also let $O_0,O_1,\ldots,O_m $ be all orbits of $H$ on $\mathcal{P}$, where $O_0=\{ \alpha \} $ and  $ \vert O_i \vert =t-1$ for every $i\in \{1,\ldots,m\}$.
Notice that $O_0\cup O_i$ has size $t$ and hence is contained in a unique block.
Let $B_1,...,B_s$ be all blocks containing one of $\{O_0\cup O_1,\ldots, O_0\cup O_m \}$.
Then $B_j$ is fixed by $H$ and hence $B_j$ is the union of some orbits of $H$ for every $j\in \{1,\ldots,s\}$.
Notice that $\alpha \in  B_{j_1} \cap B_{j_2}$ for every two distinct $j_1,j_2\in \{1,\ldots,s\}$.
If $\beta \in  (B_{j_1} \cap B_{j_2})\setminus \{ \alpha \} $, then $\beta^H \in (B_{j_1} \cap B_{j_2})$, and hence $\{ \alpha \}\cup  \beta^H $ is contained in two distinct blocks, a contradiction.
Therefore,  $B_{j_1} \cap B_{j_2} =\{ \alpha\}$.
This implies that $\mathcal{P}\setminus \{ \alpha \}=\bigcup_{i=1}^{m}O_i =\bigcup_{\ell=1}^{s}(B_\ell\setminus \{ \alpha\})$, and it follows that $v-1$ is divisible by $k-1$. 
\end{proof}


For a $G$-block-transitive $3$-$(v,k,1)$ design, it is proved in~\cite[Lemma~2.8]{GL2022+} that $(v-1)(v-2)$ divides $k(k-1)(k-2)d(d-1)$ for every subdegree $d>1$ of $G$.
The following lemma generalizes this result, and  it is convenience for computation (for example, if $G$ is almost simple, then we only need to compute the subdegrees of $T$, not of all candidates for $G$ between $T$ and $\mathrm{Aut}(T)$).

\begin{lemma}\label{lm:vvdd}
Let $\mathcal{D}=(\mathcal{P},\mathcal{B}) $ be a  $G$-block-transitive $3$-$(v,k,1)$ design, and let $T$ be a normal subgroup of $G$ such that $T$ is transitive on $\mathcal{P}$, and let $f= \vert G \vert / \vert T \vert $.
For every subdegree $d>1$ of $T$ on $\mathcal{P}$, there hold that 
\begin{enumerate}[\rm (a)]
\item $(v-1)(v-2)$  divides  $fk(k-1)(k-2)d(d-1)$, and
\item $ f\cdot \mathrm{gcd}(d(d-1),(v-1)(v-2)) >\sqrt{v}-2$.
\end{enumerate}  
\end{lemma} 

\begin{proof}  
Since $T$ is a normal subgroup of $G$ and $G$ acts transitively on $\mathcal{B}$, we see that the orbits of $T$ on $\mathcal{B}$ have the same length, and so $ \vert T_B \vert = \vert T_C \vert $ for every $B,C\in \mathcal{B}$.
Let $s$ be the number of orbits of $T$ on $\mathcal{B}$.
Then, for every $B\in \mathcal{B}$, we have $s \vert T \vert / \vert T_B \vert = \vert \mathcal{B} \vert $ and so
\[
 \vert T_B \vert =\frac{s \vert T \vert }{ \vert \mathcal{B} \vert }=\frac{s \vert G \vert /f}{ \vert G \vert / \vert G_B \vert }=\frac{s \vert G_B \vert }{f}. 
\]
Since both $G$ and $T$ are transitive on $\mathcal{P}$, we see that $ \vert G \vert / \vert G_\alpha \vert =v= \vert T \vert / \vert T_\alpha \vert $ for every $\alpha \in \mathcal{P}$, and so 
\[
 \vert T_\alpha \vert =\frac{ \vert T \vert }{v}=\frac{ \vert T \vert }{ \vert G \vert / \vert G_\alpha \vert }= \frac{ \vert G_\alpha \vert }{f}.
\]
It follows from Lemma~\ref{lm:akbv}(a) that 
\[
\frac{ \vert T_\alpha \vert }{ \vert T_B \vert }=\frac{ \vert G_\alpha \vert /f }{s \vert G_B \vert /f}=\frac{(v-1)(v-2)}{sk(k-1)(k-2)} \text{ for every } \alpha \in \mathcal{P},\ B \in \mathcal{B}.
\]  

Let $\alpha \in \mathcal{P}$ and let $ \Gamma$ be a $T_\alpha$-orbit on $\mathcal{P}$ such that $ \vert \Gamma \vert =d$. 
Let $\mathcal{B}(\alpha)=\{B: B \in \mathcal{B} \mid \alpha \in B \}$, that is, the set of blocks containing $\alpha$.
Let $ O_1,\ldots, O_m$ be all orbits of $T_\alpha$ on the set $\mathcal{B}(\alpha)$.
For every $i\in \{1,\ldots,m \}$ and for every two blocks $B,C \in O_i$, there is some $g\in T_\alpha$ such that $B^g=C$, which implies
\[
 \vert C\cap \Gamma \vert = \vert B^g\cap \Gamma^g \vert = \vert (B\cap \Gamma)^g \vert = \vert B\cap \Gamma \vert .
\]
For every $i\in \{1,\ldots,m\}$, let $B_i$ be a representation of $O_i$  and let  $\mu_i= \vert B_i \cap \Gamma \vert $.
Then 
\begin{align*}
 \vert O_i \vert &=\frac{ \vert T_\alpha \vert }{ \vert T_{\alpha B_i} \vert }=\frac{ \vert T_\alpha \vert }{ \vert T_{B_i} \vert }\cdot \frac{ \vert T_{B_i} \vert }{ \vert T_{\alpha B_i} \vert }=\frac{(v-1)(v-2)}{sk(k-1)(k-2)}\cdot \frac{ \vert T_{B_i} \vert }{ \vert T_{\alpha B_i} \vert }=\frac{(v-1)(v-2)}{k(k-1)(k-2)}\cdot \frac{ \vert T_{B_i} \vert }{s} \cdot \frac{1}{ \vert T_{\alpha B_i} \vert },\\
&=\frac{(v-1)(v-2)}{k(k-1)(k-2)}\cdot \frac{ \vert G_{B_i} \vert }{f} \cdot \frac{1}{ \vert T_{\alpha B_i} \vert }=\frac{(v-1)(v-2)}{fk(k-1)(k-2)} \cdot \frac{ \vert G_{B_i} \vert }{ \vert T_{\alpha B_i} \vert },
\end{align*}  
where $T_{\alpha B_i}=T_{\alpha } \cap T_{ B_i}$.

To prove part (a), we count the size of the set 
\[R:=\{ (B,\{\beta,\gamma\}) : B \in \mathcal{B}(\alpha),\beta,\gamma \in \Gamma \mid \beta,\gamma\in B,\ \beta \neq \gamma \}.\] 
On the one hand, for a given $2$-subset $\{\beta,\gamma\}$ of $\Gamma$, there is a unique block $B\in \mathcal{B}(\alpha)$ such that $\beta,\gamma\in B $ (by the definition of $3$-$(v,k,1)$ design), and this implies that $ \vert R \vert =d(d-1)/2$.
On the other hand, for a given $B \in \mathcal{B}(\alpha)$, we may assume that $B\in O_i$, and then there are $\mu_i(\mu_i-1)/2$ choices for $\{\beta,\gamma\} \subseteq B\cap \Gamma$.
It follows that
\[
\frac{d(d-1)}{2}= \vert R \vert =\sum_{i=1}^{m} \left(\frac{\mu_i(\mu_i-1)}{2}\cdot  \vert O_i \vert  \right)=\frac{(v-1)(v-2)}{2fk(k-1)(k-2)} \cdot \sum_{i=1}^{m} \left( \mu_i(\mu_i-1)\frac{ \vert G_{B_i} \vert }{ \vert T_{\alpha B_i} \vert } \right).
\]
Notice that $ \vert G_{B_i} \vert / \vert T_{\alpha B_i} \vert $ is an integer, since $T_{\alpha B_i}=T_\alpha \cap T_{B_i}$ is a subgroup of $G_{B_i}$.
Therefore, part (a) holds.


From part (a), we conclude that $(v-1)(v-2)$ divides $fk(k-1)(k-2)\cdot  \mathrm{gcd}(d(d-1),(v-1)(v-2))$.
This together Lemma~\ref{lm:akbv}(a)  imply that 
\[
\frac{ \vert G_\alpha \vert }{ \vert G_B \vert }=\frac{(v-1)(v-2)}{k(k-1)(k-2)} \leq f\cdot \mathrm{gcd}(d(d-1),(v-1)(v-2)).
\]
Then part (b) follows from Lemma~\ref{lm:akbv}(b),(c). 
\end{proof}
 
%
To end this section, we state the classification of large maximal subgroups of almost simple exceptional groups of Lie type.

\begin{theorem}[{\rm \cite[Theorem~1.2]{ABD2022}}]\label{th:largeAS}
Let $G$ be a finite almost simple group whose socle $T$ is a finite simple exceptional group of Lie type, and let $H$ be a maximal subgroup of $G$ not containing $T$. If $H$ is a large subgroup of $G$, then $H$ is either parabolic, or one of the subgroups listed in Table~\ref{tb:largesubgroups}.
\end{theorem}

\begin{table}[ht]
\caption{Large maximal non-parabolic subgroups  of almost simple exceptional groups  of Lie type.}\label{tb:largesubgroups}%
\begin{tabular}{@{}lll@{}}
\toprule 
$T$ & $H \cap T $ or type of $H$  & Conditions  \\
\midrule
${}^2B_2(q)\  (q=2^{2n+1}\geq 8) $ & $ (q+\sqrt{2q}+1):4 $ & $ q=8,32$  \\
  &  ${}^2B_2(q^{1/3}) $ & $q\geq 8$, $3\mid 2n+1$  \\
  
${}^2G_2(q)\  (q=3^{2n+1}\geq 27) $ & $ A_1(q) $ &      \\
  &  ${}^2B_2(q^{1/3}) $ & $3 \mid 2n+1$ \\

${}^3D_4(q) $ & $ A_1(q^3)A_1(q)$, $(q^2+ \epsilon q+1)A_2^{\epsilon}(q)$, ${}^3D_4(q^{1/2})$, $G_2(q)$ & $ \epsilon=\pm$  \\
&  ${}^3D_4(q^{1/2}) $ & $q$ square  \\
&  $7^2:\mathrm{SL}_2(3) $ & $q=2$  \\

${}^2F_4(q)'\  (q=2^{2n+1}\geq 2) $ & ${}^2B_2(q)\wr 2$, $B_2(q): 2$, ${}^2F_4(q^{1/3}) $ &    \\
&  $ \mathrm{SU}_3(q):2 $, $\mathrm{PGU}_3(q):2$ & $q=8$  \\  
&  $ \mathrm{PSL}_3(3):2 $, $\mathrm{PSL}_2(25) $, $\mathrm{A}_6.2^2$, $5^2:4\mathrm{A}_4$ & $q=2$  \\      
 $G_2(q)$ ~$(q\geq 3)$ & $ A_2^{\pm}(q)$, $A_1(q)^2$, $G_2(q^{1/r})$  & $r=2,3$   \\
 &  ${}^2G_2(q) $ & $q=3^a$, $a$ is odd  \\  
 &  $G_2(2) $ & $q=5,7$   \\  
 &  $A_1(13) $, $\mathrm{J}_2$ & $q=4$   \\  
  &  $2^3.\mathrm{SL}_3(2) $ & $q=3,5$   \\

$F_4(q)$ & $B_4(q)$, $D_4(q)$, ${}^3D_4(q)$, $F_4(q^{1/r}) $ & $r=2,3$ \\
 
 & $A_1(q)C_3(q)$ & $ p\neq 2$ \\

 & $C_4(q)$, $C_2(q^2)$, $C_2(q)^2$ & $ p= 2 $\\
 
  & ${}^2F_4(q)  $ & $ q=2^{2n+1}\geq 2$ \\

 & $A_1(q)G_2(q)$  & $ q>3  $ odd\\ 
 
 & ${}^3D_4(2)  $ & $ q=3 $\\
 
 & $\mathrm{A}_9$, $\mathrm{A}_{10}$, $\mathrm{PSL}_4(3).2 $, $\mathrm{J}_2$ & $ q=2 $\\

& $\mathrm{S}_6\wr \mathrm{S}_2$   & $ q =2$\\  

$E_6^{\epsilon}(q)$ & $ A_1(q)A_5^{\epsilon}(q)$, $F_4(q)$ &   \\
 
& $(q-\epsilon)D_5^{\epsilon} $ & $\epsilon=-$ \\
 
& $C_4(q) $ & $ p\neq 2$\\
 
& $E_6^{\pm}(q^{1/2}) $ & $ \epsilon=+$\\
 
& $ E_6^{\epsilon}(q^{1/3}) $ &  \\
 
& $ (q-\epsilon)^2. D_4(q) $ & $ (\epsilon,q)\neq (+,2)$\\
 
& $(q^2+\epsilon q+1).{}^3D_4(q) $ & $ (\epsilon,q)\neq (-,2)$\\
 
& $\mathrm{J}_{3}$, $\mathrm{A}_{12}$, $B_3(3)$, $\mathrm{Fi}_{22} $ & $ (\epsilon,q)= (-,2)$\\

$E_7(q)$& $(q-\epsilon)E_6^{\epsilon}(q)$, $A_1(q)D_6(q)$, $A_7^{\epsilon}(q)$, $A_1(q)F_4(q)$, $E_7(q^{1/r})$
& $\epsilon=\pm$ and $r=2,3$ \\

&   $\mathrm{Fi}_{22} $ &  $q=2$\\
    
$E_8(q)$& $A_1(q)E_7(q)$, $D_8(q)$, $A_2^{\epsilon}(q)E_6^{\epsilon}(q)$, $E_8(q^{1/r})$ & $\epsilon=\pm $ and $r=2,3$\\
  
\botrule 
\end{tabular} 
\end{table}

\begin{remark}\label{remark1}
\begin{enumerate}[\rm (a)] 
\item For the orders of simple groups of Lie type, see,  for example,~\cite[Tables~5.1.A--5.1.B]{K-Lie}.
 
 \item For the definition and  group structure of parabolic subgroup, one may refer to~\cite[Definition~2.6.4~and~Theorem~2.6.5]{CFSG}.

\item In Table~\ref{tb:largesubgroups}, the type of $H$  is an approximate description of the group-theoretic structure of $H$, and for precise structure, one may refer to~\cite[Tables~8.30,~8.41~and~8.42]{BHRD2013} for $G_2(q)$, and~\cite[Tables~5.1]{CLS1992} for $F_4(q)$, $E_6(q)$ and ${}^2E_6(q)$, $E_7(q)$ and $E_8(q)$. 

\item According to~\cite[p.303]{BAtlas}, the list of maximal subgroups of $F_4(2)$ in Atlas~\cite{Atlas} is complete.
Hence, the candidates $\mathrm{A}_9$, $\mathrm{A}_{10}$ and $\mathrm{J}_2$ for $F_4(2)$ in Table~\ref{tb:largesubgroups} do not happen. 
\end{enumerate} 
\end{remark}  

\smallskip


\section{Proof of Theorem~\ref{th:exceptional}} \label{sec:proof}
Throughout this section, we  make the following hypothesis.

\begin{hypothesis}\label{hy:1}
Let $G$ be an almost simple group whose  socle $T$ is a simple exceptional group of Lie type, and  let $\mathcal{D}=(\mathcal{P},\mathcal{B})$ be a non-trivial $G$-block-transitive $3$-$(v,k,1)$ design. 
Let  $\alpha $ be a point in $\mathcal{P}$ and let $B$ be a block in $\mathcal{B}$.
Let $q=p^e$, where $p$ is a prime and $e$ is an positive integer.
Let $G=T.\mathcal{O}$, where $\mathcal{O}\leq \mathrm{Out} (T)$.
\end{hypothesis}

Since $\mathcal{D}$ is non-trivial, we have $3  <k<v$.
By Lemma~\ref{lm:VP},  $G$ acts primitively on $\mathcal{P}$.
Thus,  $G_\alpha$ is a core-free  maximal subgroup of $G$ and  the normal subgroup $T$ is transitive on $\mathcal{P}$.
From the transitivity of $T$, we have  $v= \vert G \vert / \vert G_\alpha \vert = \vert T \vert / \vert T_\alpha \vert $
 and $G_\alpha=T_\alpha.\mathcal{O}$.

\begin{lemma}\label{lm:Ga3>G}
$ \vert G_\alpha \vert ^3> \vert G \vert $.
\end{lemma}

\begin{proof}
Recall Lemma~\ref{lm:akbv}(b), that is, $  \vert G_{\alpha} \vert (  \vert G_{\alpha} \vert / \vert G_{B} \vert  +2)^2> \vert G \vert $. 

We claim that $ \vert G_B \vert \geq 2$. Assume at first that $T \neq  {}^2B_2(q)$. Then $ \vert T \vert $ is divisible by $(q+1)q(q-1)$ (see~\cite[Table~5.1.B]{K-Lie} for example) and hence by $3$. 
Let $g \in T$ be of order $3$. 
Since $G$ acts faithfully on $\mathcal{P}$, we see that there exists some $\beta \in \mathcal{P}$ which is not fixed by $g$. 
Then $\langle g\rangle$ fixes the $3$-subset  $\{\beta,\beta^g,\beta^{g^2} \}$ of $\mathcal{P}$ and hence fixes the unique block  containing $\{\beta,\beta^g,\beta^{g^2} \}$ by the definition of $3$-$(v,k,1)$ design.
It then follows from the block-transitivity of $G$ that $ \vert G_B \vert \geq 3 $. 
Assume now that $T ={}^2B_2(q)$. From~\cite[Table~8.16]{BHRD2013} we see that $ \vert T_\alpha \vert $ is divisible by $2$.
Let $g\in T_\alpha$ be of order $2$ and let $\beta \in \mathcal{P}\setminus \{ \alpha\}$  which is not fixed by $g$. 
Then $g$ fixes the unique block containing $\{ \alpha,\beta,\beta^g\}$, and hence $ \vert G_B \vert \geq 2$. Therefore, the claim holds.

Suppose that $ \vert G_\alpha \vert \leq 4$. Since $ \vert G_B \vert \geq 2$, we have  $ \vert G \vert < \vert G_\alpha \vert (  \vert G_{\alpha} \vert / \vert G_{B} \vert  +2)^2<4^3=64$, which implies $G=\mathrm{A}_5$. 
Notice that the maximal subgroups of $G=\mathrm{A}_5$ are $\mathrm{A}_4 $, $\mathrm{S}_3 $ and $\mathrm{D}_{10} $. 
Since $G_\alpha$ is maximal in $G$, we have $ \vert G_\alpha \vert \geq 6$.
Now $ \vert G_{\alpha} \vert / \vert G_{B} \vert  \leq   \vert G_{\alpha} \vert /2 $ and $2<  \vert G_{\alpha} \vert /2 $.
Then $   \vert G_{\alpha} \vert / \vert G_{B} \vert  +2 < \vert G_{\alpha} \vert  $, and the lemma follows.  
\end{proof}

We remark that Lemma~\ref{lm:Ga3>G} holds for all almost simple groups, since the Suzuki groups ${}^2B_2(q)$ are the only non-abelian simple groups whose order is not divisible by $3$.

\medskip

By Lemma~\ref{lm:Ga3>G},  all candidates for $(T,T_\alpha)$  have been determined in Theorem~\ref{th:largeAS}.

\subsection{Computation methods}\label{sec:methods}

Our methods based on the  following lemma, which slightly improves  Lemma~\ref{lm:akbv}(c).
Recall that $ \vert G_\alpha \vert = \vert T_\alpha \vert \cdot \vert \mathcal{O} \vert $.

\begin{lemma}\label{lm:XGCD}
$\mathrm{gcd}( \vert T_\alpha \vert ,(v-1)(v-2)) \vert \mathcal{O} \vert  >\sqrt{v}-2$.
\end{lemma}
\smallskip

Lemma~\ref{lm:XGCD} can help us quickly test whether a candidate for $(T,T_\alpha)$ is feasible.  
For example, let $T=G_2(q)$ and let $T_\alpha =[q^5]\, {: }\,\mathrm{GL}_2(q)$ be a maximal parabolic subgroup. 
Then $ \vert T \vert =q^6(q^6-1)(q^2-1)$, and $  \vert \mathcal{O} \vert $ divides $e$, and $ \vert T_\alpha \vert =q^5 \vert \mathrm{GL}_2(q) \vert =q^6(q^2-1)(q-1)$, and  $v= \vert G \vert / \vert G_\alpha \vert =(q^6-1)/(q-1)$, and so 
\[
(v-1)(v-2)=q^{10} + 2q^9 + 3q^8 + 4q^7 + 5q^6 + 3q^5 + 2q^4 + q^3 - q.
\]
View $ \vert T_\alpha \vert $  and $(v-1)(v-2)$ as polynomials in $q$ over the rational field. 
Using the {\sc Magma}~\cite{Magma} command \textsf{XGCD} (or the  {\sc Gap}~\cite{GAP4} command \textsf{GcdRepresentation}), we can obtain  three rational polynomials  $R(q)$, $P(q)$, $Q(q)$ such that 
\[P(q) \vert T_\alpha \vert +Q(q)(v-1)(v-2)=R(q).\]
For this example, computation in {\sc Magma}~\cite{Magma} shows that 
\begin{align*}
R(q)&=q,\\ 
P(q)&=\frac{357}{32}q^8 + \frac{4431}{160}q^7 + \frac{3787}{80}q^6 + \frac{1093}{16}q^5 + \frac{2891}{32}q^4 + \frac{1267}{16}q^3 + \frac{1261}{20}q^2 + \frac{6913}{160}q + \frac{3589}{160},\\
Q(q)&=-\frac{357}{32}q^7 + \frac{231}{40}q^6 + \frac{2149}{160}q^5 - 4q^4 - 2q^3 - q^2 - 1.
\end{align*}
Notice that the least common multiple of the denominators of all coefficients of $P(q)$ and $R(q)$ is $160$.
Therefore,
\[R_1(q)=P_1(q) \vert T_\alpha \vert +Q_1(q)(v-1)(v-2),\]
 where $R_1(q)=160\cdot R(q)$, $P_1(q)=160\cdot P(q)$ and $Q_1(q)=160\cdot Q(q)$.
 This implies  
 \[
 \mathrm{gcd}( \vert T_\alpha \vert ,(v-1)(v-2)) \mid R_1(q)=160q.
 \] 
To obtain $R_1(q),P_1(q),Q_1(q)$  above, one may use the following {\sc Magma}~\cite{Magma} code:
\begin{verbatim}
F<q>:=PolynomialRing(RationalField());  
Ta:=q^6*(q^2-1)*(q-1); v:=(q^6-1) div (q-1);
R,P,Q:=XGCD(Ta,(v-1)*(v-2));
coes:={i: i in Eltseq(P)} join {i: i in Eltseq(Q)};
dens:={Denominator(i): i in coes}; 
R1:=LCM(dens)*R; P1:=LCM(dens)*P;  Q1:=LCM(dens)*Q;   
\end{verbatim} 
Then, from Lemma~\ref{lm:XGCD} we conclude that
\begin{align*}
160q\log_p(q)>(q^5 + q^4 + q^3 + q^2 + q + 1)^{\frac{1}{2}}-2.
\end{align*} 
This is an inequality on $q$, and one can show that it holds only for some $q\leq 128$.


\vspace{3mm}

For a certain  $(T,T_\alpha)$, one may check whether  exists a non-trivial $G$-block-transitive $3$-$(v,k,1)$ design by the following steps.
\begin{enumerate}[\rm (i)]
\item Check whether 
\begin{equation}\label{eq:k-1}
\mathrm{gcd}( \vert T_\alpha \vert ,(v-1)(v-2)) \vert \mathrm{Out} (T) \vert >\sqrt{v}-2.
\end{equation}
\item Find divisors of $v-2$ (one may use the  {\sc Magma} command \textsf{Divisors}), which gives candidates for $k-2$ as
\begin{equation}
\label{eq:k0}
(k-2) \mid (v-2),
\end{equation} 
 by Lemma~\ref{lm:brlambda2}(c).
\item For every candidates for $k$, check whether the following conditions are satisfied: 
\begin{align} 
\label{eq:k1}
&k < \sqrt{v}+2,\\
\label{eq:k2}
&(k-1)(k-2) \mid (v-1)(v-2),\\ 
\label{eq:k3}
&k(k-1)(k-2) \mid v(v-1)(v-2),\\   
\label{eq:k4}
&(v-1)(v-2) \mid  k(k-1)(k-2) \vert T_\alpha \vert \cdot  \vert \mathrm{Out} (T) \vert .
\end{align}
Note that ~\eqref{eq:k1} is due to Lemma~\ref{lm:vk}, and ~\eqref{eq:k2} and \eqref{eq:k3} are by Lemma~\ref{lm:brlambda2}, and ~\eqref{eq:k4} is a consequence of Lemma~\ref{lm:akbv}(a). 

\item  For a possible $k$ satisfying ~\eqref{eq:k0}--\eqref{eq:k4}, we compute candidates for $G_B$, and the orbits of $G_{B}$ on $\mathcal{P}$ (note that $B$ consists of some orbits of $G_B$).
For every candidate for  block $C$, it holds that $C=\mathit{\Omega}_1 \cup \cdots \cup  \mathit{\Omega}_m$, where $\mathit{\Omega}_i$ for $i\in \{1,\ldots,m\}$ are $G_{B}$-orbits such that $k= \vert \mathit{\Omega}_1 \vert +\cdots + \vert  \mathit{\Omega}_m \vert $. 
If the stabilizer of $C$ in $G$ has order $ \vert G_{B} \vert $ (which implies that $ \vert C^G \vert = \vert \mathcal{B} \vert $), and, for every two distinct points $\beta,\gamma \in \mathcal{P}\setminus \{\alpha\}$, the $3$-subset $\{\alpha,\beta,\gamma \}$ is contained in exactly one block in $C^G$ (notice that $G$ is transitive on $\mathcal{P}$), then we find  a $G$-block-transitive $3$-$(v,k,1)$ design. 
\end{enumerate}
 
\subsection{Some special candidates}   
Recall that all  candidates for $(T,T_\alpha)$ has been determined by~Theorem~\ref{th:largeAS}, except for $T= G_2(2)'$ and ${}^2G_2(3)'$. 
We deal with some special candidates in this subsection.
 
Firstly, it is convenience to rule out some certain  $(T,T_\alpha)$ in Table~\ref{tb:largesubgroups} of Theorem~\ref{th:largeAS}.


\begin{lemma}\label{lm:smallcases}
$  T\notin \{ G_2(2)',{}^2G_2(3)', {}^2F_4(2)'\}$, and $(T,T_\alpha)$ is not one of the following: 
\begin{align*}
& ({}^2B_2(8),5\,{:}\,4),  ({}^2B_2(32),41\,{:}\,4), ({}^3D_4(2),7^2.\mathrm{SL}_2(3)), ({}^2F_4(8),\mathrm{SU}_3(8)\,{:}\,2), ({}^2F_4(8),\mathrm{PGU}_3(8)\,{:}\,2),  \\ 
&(G_2(3),\mathrm{PSL}_2(13)), (G_2(3),2^3.\mathrm{PSL}_3(2)),(G_2(4),\mathrm{J}_ 2),(G_2(4),\mathrm{PSL}_2(13)), (G_2(3),2^3.\mathrm{PSL}_3(2)),\\ 
& (G_2(5),G_2(2)),(G_2(7),G_2(2)), (G_2(11),\mathrm{J}_1),(F_4(2),\mathrm{PSL}_4(3).2),  (F_4(2),\mathrm{S}_6\wr \mathrm{S}_2),   (F_4(3),{}^3D_4(2)), \\
&({}^2E_6(2),\mathrm{J}_{3}), ({}^2E_6(2),\mathrm{A}_{12}), ({}^2E_6(2),\mathrm{P\Omega}_{7}(3)),({}^2E_6(2),\mathrm{Fi}_{22}),(E_7(2),\mathrm{Fi}_{22}).   
\end{align*}
%
%
%
%


\end{lemma}
\begin{proof}
If $T\in \{ G_2(2)',{}^2G_2(3)', {}^2F_4(2)'\}$ or $(T,T_\alpha)$ is one of the above, then a  straightforward calculation shows that there is no $k$ satisfying~\eqref{eq:k0}--\eqref{eq:k4}.
\end{proof}

 

 
Recall the inversive plane of order $q$  (a $3$-$(q^2,q+1,1)$ design) on the Suzuki group ${}^2B_2(q)$ in Section~\ref{sec:1}. 
In this $3$-$(q^2,q+1,1)$ design,  the group ${}^2B_2(q)$ acts $2$-transitively on the point set, and this $2$-transitive action of  ${}^2B_2(q)$ can be realized on the set of so-called \emph{Suzuki-Tits ovoids} (for definition and more information about the Suzuki-Tits ovoid, the reader may refer to Alavi's note~\cite{Alavi-note} and references therein).
Moreover, the point stabilizer in ${}^2B_2(q)$ is a maximal parabolic subgroup isomorphic to $[q^2]\,{:}\,(q-1)$. We show in the next lemma that $\mathcal{D}$ is exactly this inversive plane of $q$ when $T={}^2B_2(q)$ and $T_\alpha$ is a maximal parabolic subgroup $[q^2]\,{:}\,(q-1)$.
\begin{lemma}\label{lm:2B2q}
Suppose that $T={}^2B_2(q)$ and $T_\alpha$ is a maximal parabolic subgroup $[q^2]\,{:}\,(q-1)$, where $q=2^e$ with odd $e>1$. Then $v=q^2+1$ and $k=q+1$, and so $\mathcal{D}$ is an  inversive plane of order $q$ with $\mathrm{Aut}(\mathcal{D})=\mathrm{Aut}({}^2B_2(q))={}^2B_2(q)\,{:}\, e$.
\end{lemma}

\begin{proof}
Since $ \vert T \vert =q^2(q^2+1)(q-1)$, 
it follows from the point-transitivity of $T$ that $v= \vert T \vert / \vert T_\alpha \vert =q^2+1$.
By L\"{u}neburg~\cite[Theorem~1~and~Corollary~1]{L1964}, there is only one (up to isomorphism) inversive plane of order $q$ that admits an automorphism group isomorphic to ${}^2B_2(q)$; and  by Dembowski~\cite[p.275,4]{Dembowski-book}, the full automorphism group of this inversive plane is $\mathrm{Aut}({}^2B_2(q))={}^2B_2(q)\,{:}\,e$.
Therefore, we only need to show  $k=q+1$.

Write $e=2m+1$ with a positive integer $m$, and write $G={}^2B_2(q).f$ with $f\mid e $.   
Note that $T$ acts $2$-transitively on $\mathcal{P}$,  and so $T_\alpha=[q^2]:(q-1)$ acts transitively on $ \mathcal{P}\setminus \{\alpha\}$.
Moreover,  the normal $2$-subgroup with order $q^2$ of $T_\alpha$ acts regularly on $ \mathcal{P}\setminus \{\alpha\}$ (see~\cite[Section~7.7,~p.250]{DM-book} for example).
Let $H$ be a subgroup of order $2$ in $T_\alpha$.
Then $H$ fixes exactly one point $\alpha$ of $\mathcal{P}$ and acts semiregularly on the set of remaining points, that is,  $H$ is quasi-semiregular.
It follows from Lemma~\ref{lm:order3}(c)  that $k-1$ divides $v-1=q^2=2^{2e}$.
Therefore, $k=2^a+1$ for some integer $a$ such that $0\leq a \leq 2e$. 
Since $\mathcal{D}$ is non-trivial, we have $3< k< v$ and hence $1< a<2e$.
Moreover, since $k-2=2^a-1$ dividing $v-2=2^{2e}-1$, we conclude from  Lemma~\ref{le:divide} that $a$ divides $2e$.  
Now, by Lemma~\ref{lm:vk}, 
\[k\leq \sqrt{v}+2= \sqrt{q^2+1}+2<\sqrt{q^2+2q+1}+2=q+3. \]
This together with $q\geq 8$ implies that  $a\leq e$.

According to Lemma~\ref{lm:akbv}(a),  
\[
 \vert G_B \vert =\frac{ \vert G_\alpha \vert k(k-1)(k-2)}{(v-1)(v-2)}=\frac{fq^2(q-1)k(k-1)(k-2)}{q^2(q^2-1)}=\frac{f2^a(2^{2a}-1)(2^{e}-1)}{2^{2e}-1}. 
\]
Suppose that $e=3$. 
From $a \mid 2e=6$ and $1<a\leq e=3$, we see  $a=2$ or $3$. 
If $a=3$, then $k=2^a+1=q+1$, as required.
It remains to show that the case $a=2$ is impossible. 
Suppose for a contradiction that $a=2$. 
Now $v=2^{2e}+1=65$ and $k=2^a+1=5$.
From $ \vert G_B \vert =f\cdot 2^2\cdot 15\cdot 7/63$ and $f\mid e=3$, we conclude that $f=3$ and $ \vert G_B \vert =60$.
Then $G={}^2B_2(8)\,{:}\,3$.
Computation in {\sc Magma}~\cite{Magma} shows that $G$ has only one conjugate classes of subgroups of order $60$, and a subgroup $H$ of order $60$ has two orbits of lengths $5$ and $60$ on $\mathcal{P}$.
This implies that $\mathcal{D}$ is flag-transitive and $\mathcal{B}=C^G$, where $C$ is the orbit of length $5$ of $H$.
However, further computation shows that there is some $3$-subset of $\mathcal{P}$ which is not in a  block in $\mathcal{B}=C^G$, a contradiction. 

Therefore, $e\geq 5$.
By Lemma~\ref{lm:ppd}, $2^{2e}-1$ has a primitive prime divisor $r$, and $r $ does not divides  both $f$ and $2^{e}-1$.
Clearly, $r$ also does not divides $2^a$ as $a\leq e$.
Since  $ \vert G_B \vert $ is an integer, it follows that $r$ divides $2^{2a}-1 $, which together with  $a\leq e$ implies $a=e$.
Therefore, $k=2^a+1=q+1$, as required.  
\end{proof}

\begin{lemma}\label{lm:2G2q}
If $T={}^2G_2(q)$, where $q=3^e$ with odd $e>1$, then $T_\alpha$ is not a maximal parabolic subgroup $[q^3]\,{:}\,(q-1)$.
\end{lemma}

 \begin{proof} 
 Suppose for a contradiction that $T_\alpha=[q^3]\,{:}\,(q-1)$.
 By~\cite[Table~8.43]{BHRD2013} we see that $G={}^2G_2(q).f$, where $f \mid e$, and $G_\alpha=T_\alpha.f$. 
Now $ v= \vert T \vert / \vert T_\alpha \vert =q^3+1$.  
Notice that $T$ acts $2$-transitively on $\mathcal{P}$,  and the normal $3$-subgroup with order $q^3$ of $T_\alpha$ acts regularly on $ \mathcal{P}\setminus \{\alpha\}$ (see~\cite[p.252]{DM-book} for example).
Therefore, every subgroup of order $3$  in $T_\alpha$ is quasi-semiregular.
From Lemma~\ref{lm:order3}(c)  we see that $v-1=q^3=3^{3e}$ is divisible by either $k-1$  or $k$.

Assume that $v-1 $ is divisible  by $k-1$. 
Then $k=3^a+1$ for some $0\leq a \leq 3e$.
Since $\mathcal{D}$ is non-trivial, we have $k\geq 4$ and so $1\leq a < 3e$.
By Lemma~\ref{le:divide}, $a$ divides $3e$, which implies that $2a < 3e$.
From Lemma~\ref{lm:akbv}, we conclude 
\[
 \vert G_B \vert =\frac{ \vert G_\alpha \vert k(k-1)(k-2)}{ (v-1)(v-2)}=\frac{ q^3 (q-1)f(3^a+1)3^a(3^a-1)}{q^3(q^3-1)}=\frac{(3^e-1)f3^a(3^{2a}-1)}{ 3^{3e}-1  }.
\]
However, Lemma~\ref{lm:ppd} tells us that $3^{3e}-1$ has a primitive prime divisor $r$ such that $r$ does not divide $f$, $3^e-1$ and $3^{2a}-1$ (notice that $2a<3e$).
This is a contradiction.

Assume that $v-1 $ is divisible  by $k$. 
Then $k=3^a$ for some $2\leq a \leq 3e$ (notice that $k\geq 4$). 
Now, 
\[
 \vert G_B \vert =\frac{ \vert G_\alpha \vert k(k-1)(k-2)}{ (v-1)(v-2)}=\frac{ q^3 (q-1)fk(k-1)(k-2)}{q^3(q^3-1)}=\frac{(3^{e}-1)f3^a(3^a-1)(3^a-2)}{  3^{3e}-1  }.
\]
Since $\mathrm{gcd}(3^{3e}-1,3^a)=1$, we conclude that $ \vert G_B \vert $ is divisible by $3^a$, and hence
\begin{equation}\label{eq:2G2q}
 3^{3e}-1 \mid (3^{e}-1)f(3^a-1)(3^a-2).
\end{equation} 

According to Lemma~\ref{lm:brlambda2}, $(v-1)(v-2)=3^{3e}(3^{3e}-1)$ is divisible by $(k-1)(k-2)= (3^{a}-1)(3^{a}-2)$.
Since $\mathrm{gcd}(3^{a}-1,3^{3e})=1$ and $\mathrm{gcd}(3^{a}-2,3^{3e})=1$, we conclude that $3^{3e}-1$  is divisible by $(3^{a}-1)(3^{a}-2)$.
In particular, $3^{3e}-1$  is divisible by $ 3^{a}-1 $.
It follows from Lemma~\ref{le:divide} that $a \mid 3e$.  
Let $m$ be an integer such that $am=3e$.
Since $e$ is odd, we see that $m$ is also odd and $a\leq e$.

Suppose that $a<e$. Then $m>3$ and so $m\geq 5$. From~\eqref{eq:2G2q} we conclude that
\[
  3^{3e-1}<3^{3e}-1\leq f\cdot (3^{e}-1)(3^{\frac{3e}{5}}-1)(3^{\frac{3e}{5}}-2)<e \cdot  3^{\frac{11}{5}e},
\] 
and hence $ 3e-1<\log_3(e)+11e/5$. This inequality does not hold for any odd $e\geq 3$. Therefore, the case $a<e$ is impossible.

Suppose that $a=e$. Since $(3^{3e}-1)-(3^{2e}+2\cdot 3^e+4)(3^e-2)=7$, we see that $\mathrm{gcd}(3^{3e}-1,3^e-2)\mid 7$.
Then from~\eqref{eq:2G2q} we conclude that
\[
3^{3e-1}<3^{3e}-1\leq f\cdot (3^{e}-1)^2 \cdot 7<3^2\cdot e \cdot 3^{2e},
\]
which implies that $3e-1<2+\log_3(e)+2e$.
Notice that $e=3$ is the only solution for this inequality on odd $e\geq 3$.
However, now $757\mid 3^{3\cdot 3}-1 $, while $757$ is not a divisor of $ (3^3-1)^2\cdot 3\cdot (3^3-2)$, contradicting~\eqref{eq:2G2q}. 
 \end{proof}

\begin{lemma}\label{lm:E6q}
If $T=E_6(q)$, then  $T_\alpha$ is  not  a maximal parabolic subgroup of type $D_5(q)$. 
\end{lemma}

\begin{proof}
Suppose for a contradiction that $T_\alpha$ is a maximal parabolic subgroup of type $D_5(q)$.
From~\cite[Theorem~1]{Vasilev2} we see that  $
v=  (q^{8}+q^4+1)(q^9-1)/(q-1)$ 
and $T_\alpha$ has $2$ orbits on $\mathcal{P}\setminus \{\alpha\}$ with lengths $d_1:= (q^8-1)(q^3+1)/(q-1)$ and $d_2:=q^8(q^5-1)(q^4+1)/(q-1)$.
Since $|\mathrm{Out}(T)|=\mathrm{gcd}(3,q-1)\log_p(q)$, it follows that $f:= \vert G \vert / \vert T \vert $ divides $ 3\log_p(q)$. 
Let $d=d_2$.
Computation in {\sc Magma}~\cite{Magma} shows that $\mathrm{gcd}(d(d-1),(v-1)(v-2)) $ divides $ 7970(q^5+1)$.
According to Lemma~\ref{lm:vvdd}(b), we derive that
\[
3\log_p(q)\cdot 7970(q^5+1) <\sqrt{v}-2.
\]
Computation shows that the above inequality on $q$ holds only if $q\leq 32$.
Further computation shows that, for any $q\leq 32$, there is no $k$ satisfying the above inequality and \eqref{eq:k0}--\eqref{eq:k4}. 
\end{proof}

\begin{lemma}\label{lm:E7q}
If $T=E_7(q)$, then  $T_\alpha$ is  not  a maximal parabolic subgroup of type $E_6(q)$. 
\end{lemma}

\begin{proof}
Suppose for a contradiction that $T_\alpha$ is a maximal parabolic subgroup of type $E_6(q)$.
From~\cite[Theorem~2]{Vasilev2} we see that  $
v=  (q^{14}-1)(q^9+1)(q^5+1)/(q-1)$ and $T_\alpha$ has one orbit on $\mathcal{P}\setminus \{\alpha\}$ with length $d=q^{27}$. 
Since $|\mathrm{Out}(T)|=\mathrm{gcd}(2,q-1)\log_p(q)$, it follows that $f:= \vert G \vert / \vert T \vert $ divides $ 2\log_p(q)$. 
Computation in {\sc Magma}~\cite{Magma} shows that $\mathrm{gcd}(d(d-1),(v-1)(v-2)) $ divides $ 990q(q^9-1)/(q-1)$.
According to Lemma~\ref{lm:vvdd}(b), we derive that
\[
2\log_p(q)\cdot 990q(q^9-1)/(q-1) <\sqrt{v}-2.
\]
Computation shows that the above inequality on $q$ holds only if $q\leq 5$.
However, for any $q\leq 5$, there is no $k$ satisfying the above inequality and \eqref{eq:k0}--\eqref{eq:k4}. 
\end{proof}

\begin{lemma}\label{lm:G2q}
If $T=G_2(q)$, where $q\geq 3$, and $T_\alpha$ is  a maximal subgroup of $T$ isomorphic to $\mathrm{SU}_3(q).2$ or $\mathrm{SL}_3(q).2$, then  $q>10^5$. 
\end{lemma}

\begin{proof}
Note that $ \vert T \vert =q^6(q^6-1)(q^2-1)$. 
By~\cite[Tables~8.30,~8.41~and~8.42]{BHRD2013} we see that $G=T.f$, where $f \mid e$, and $G_\alpha=T_\alpha.f$. Recall Lemma~\ref{lm:akbv}(a), that is, 
\[ \vert G_B \vert (v-1)(v-2) =  \vert G_\alpha \vert k(k-1)(k-2)= \vert T_\alpha \vert fk(k-1)(k-2).\]
It follows that $k(k-1)(k-2)f$ is divisible by $(v-1)(v-2)/\mathrm{gcd}( \vert T_\alpha \vert ,(v-1)(v-2))$.

Assume first that $T_\alpha=\mathrm{SU}_3(q).2$.
Then $v= \vert T \vert / \vert T_\alpha \vert =q^3(q^3-1)/2$. 
By computation  in {\sc Magma}~\cite{Magma}, $\mathrm{gcd}( \vert T_\alpha \vert ,(v-1)(v-2)) \mid 48(q^3+1)$, which implies that
\[
k(k-1)(k-2)f \geq \frac{(v-1)(v-2)}{48(q^3+1)}=\frac{(q^3-2)(q^6-q^3-4)}{192}.
\]
A straightforward calculation shows that,  for any prime power $q\leq 10^5$, there is no $k$ satisfies the above inequality and \eqref{eq:k0}--\eqref{eq:k4}.
 
Assume now that $T_\alpha=\mathrm{SL}_3(q).2$. 
In this case, we have $v=q^3(q^3+1)/2$, and $\mathrm{gcd}( \vert T_\alpha \vert ,(v-1)(v-2)) \mid 48(q^3-1)$, and so
\[
k(k-1)(k-2)f \geq \frac{(v-1)(v-2)}{48(q^3-1)}=\frac{(q^3+2)(q^6+q^3-4)}{192}.
\]
Also, computation shows that those $q\leq 10^5$ are impossible.  
\end{proof}


\noindent\textit{Proof of Theorem~\ref{th:exceptional}}. 
By Lemma~\ref{lm:Ga3>G}, all candidates for $(T,T_\alpha)$ are given in Theorem~\ref{th:exceptional}. 
Notice that parts (a) and (b) of Theorem~\ref{th:exceptional} arise  from Lemma~\ref{lm:2B2q} and Lemma~\ref{lm:G2q}, respectively.

According to Lemmas~\ref{lm:smallcases}--\ref{lm:G2q}, we may assume that $(T,T_\alpha)$ does not satisfy the following.
\begin{enumerate}[\rm (i)]
\item $T \in \{G_2(2)',{}^2G_2(3)', {}^2F_4(2)'\}$, or $(T,T_\alpha)$ is one of those in  Lemma~\ref{lm:smallcases}.
\item $T={}^2B_2(q)$  and $T_\alpha=[q^2]\,{:}\,(q-1)$ is a maximal parabolic subgroup.
\item $T={}^2G_2(q)$  and $T_\alpha=[q^3]\,{:}\,(q-1)$ is a maximal parabolic subgroup. 
\item $T=E_6(q)$  and $T_\alpha$ is a maximal parabolic subgroup of type $D_5(q)$.
\item $T=E_7(q)$  and $T_\alpha$ is a maximal parabolic subgroup of type  $E_6(q)$.
\item $T=G_2(q)$ and $T_\alpha=\mathrm{SL}_3(q).2$ or $\mathrm{SU}_3(q).2$  (the group of type $A_2^{+}(q)$ or $A_2^{-}(q)$ in Table~\ref{tb:largesubgroups}, respectively).
\end{enumerate}

For the remaining candidates,  by computation in {\sc Magma}~\cite{Magma}  using the method in Subsection~\ref{sec:methods}, we can obtain three  polynomials $R_1(q)$, $P_1(q)$ and $Q_1(q)$  such that their coefficients  are integers and
\[R_1(q)=P_1(q) \vert T_\alpha \vert +Q_1(q)(v-1)(v-2).\]
Then  $\mathrm{gcd}( \vert T_\alpha \vert ,(v-1)(v-2))$ divides $R_1(q)$, and from Lemma~\ref{lm:XGCD} we obtain an inequality 
\[
R_1(q)  \vert \mathrm{Out} (T) \vert  >\sqrt{v}-2.
\] 
Computation shows that the inequality on $q$ has only solutions for some $q\leq 10^4$, and for every solution,  there is no possible $(v,k)$ satisfies \eqref{eq:k-1}--\eqref{eq:k4}. 
Therefore, these remaining candidates for $(T,T_\alpha)$ are impossible. \hfill\qed

\bmhead{Acknowledgments} 
We acknowledge the algebra software package {\sc Magma}~\cite{Magma} for computation. 
This work was supported by the National Natural Science Foundation of China (12071484, 12271524, 12071023, 11971054).

\section*{Declarations} 

\begin{itemize}
\item {\bf Conflict of interest} 
The authors declare they have no financial interests.
\item {\bf Availability of data and materials} 
Data sharing not applicable to this article as no datasets were generated or analysed during the current study.
\end{itemize}


\begin{thebibliography}{99}
 
 
 



 \bibitem{Alavi-note}
S.H.~Alavi, A note on two families of $2$-designs arose from Suzuki-Tits ovoid, \emph{Algebra and Discrete mathematics} 34 (2022),  169--175. \url{http://dx.doi.org/10.12958/adm1687}
 
 \bibitem{ABD2022}
S.H. Alavi, M. Bayat, A. Daneshkhah,
Finite exceptional groups of Lie type and symmetric designs,
\emph{Discrete Mathematics} 345 (2022), 112894. 
\url{https://doi.org/10.1016/j.disc.2022.112894}

 
 \bibitem{AB2015}
 S.H.~Alavi, T.C.~Burness, Large subgroups of simple groups, \emph{Journal of Algebra}  421 (2015), 187--233. \url{https://doi.org/10.1016/j.jalgebra.2014.08.026}
 
  
  
  
  \bibitem{C1976}
P.J.~Cameron, \emph{Parallelisms of Complete Designs},  Cambridge University Press, New York, 1976. \url{https://doi.org/10.1017/CBO9780511662102} 


\bibitem{CP1993}
P.J.~Cameron, C.E.~Praeger, Block-transitive $t$-designs I: point-imprimitive designs, \emph{Discrete Mathematics} 118 (1993), 33--43. \url{https://doi.org/10.1016/0012-365X(93)90051-T} 

\bibitem{CP1993lagert}
P.J.~Cameron, C.E.~Praeger, Block-transitive $t$-designs, II: large $t$, in: F.~Clerck, J.~Hirschfeld (Eds.), \emph{Finite Geometry and Combinatorics}, Cambridge University Press,  Cambridge, 1993, 103--119.
\url{https://doi.org/10.1017/CBO9780511526336.012} 

\bibitem{CGZ2008}
A.R.~Camina, N.~Gill, A.E.~Zalesski, Large dimensional classical groups and linear spaces, \emph{Bulletin of the Belgian Mathematical Society-Simon Stevin} 15 (2008), 705--731.
\url{https://doi.org/10.36045/BBMS/1225893950} 
 

\bibitem{CNP2003}
A.R.~Camina, P.M.~Neumann, C.E.~Praeger, Alternating groups acting on finite linear spaces, \emph{Proceedings of the London Mathematical Society} 87 (2003), 29--53.
\url{https://doi.org/10.1112/S0024611503014060} 
 

 

\bibitem{CP2001}
A.R.~Camina, C.E.~Praeger, Line-transitive, point quasiprimitive automorphism groups of finite linear spaces are affine or almost simple, \emph{Aequationes Mathematicae} 61 (2001), 221--232.
\url{https://doi.org/10.1007/s000100050174}  
 
  \bibitem{CS2000}
A.~Camina, F.~Spiezia, Sporadic groups and automorphisms
of linear spaces, \emph{Journal of Combinatorial Designs} 8 (2000), 353--362.
\url{https://doi.org/10.1002/1520-6610(2000)8:5<353::AID-JCD5>3.0.CO;2-G}  
  

%
%
%


\bibitem{Magma}
W.~Bosma, J.~Cannon, C.~Playoust, The MAGMA algebra system I: The user language, \emph{Journal of Symbolic Computation} 24 (1997), 235--265.
\url{https://doi.org/10.1006/jsco.1996.0125}  
 
 
 \bibitem{BHRD2013}
J.N.~Bray, D.F.~Holt, C.M.~Roney-Dougal, \emph{The maximal subgroups of the low-dimensional finite classical group}, Cambridge University Press, New York, 2013.
\url{https://doi.org/10.1017/CBO9781139192576}  


\bibitem{BDDKLS1990}
F.~Buekenhout, A.~Delandtsheer, J.~Doyen, P.B.~Kleidman, M.~Liebeck, J.~Saxl, Linear spaces with flag-transitive automorphism groups, \emph{Geometriae Dedicata} 36 (1990), 89--94.
\url{https://doi.org/10.1007/BF00181466}  
 
 
 \bibitem{Atlas}
J.H.~Conway, R.T.~Curtis, S.P.~Norton, R.A.~Parker,  R.A.~Wilson, \emph{Atlas of Finite Groups}, Oxford University Press, New York, 1985.


%



%
%
%
%
 \bibitem{Dembowski-book}
 P.~Dembowski, \emph{Finite Geometries}, Springer Berlin, Heidelberg, 1968. \url{https://doi.org/10.1007/978-3-642-62012-6} 

\bibitem{DM-book}
J.D.~Dixon and B.~Mortimer, \emph{Permutation Groups}, Springer-Verlag, New York, 1996.
\url{https://doi.org/10.1007/978-1-4612-0731-3}
 
 

\bibitem{GL2022+}
Y.S.~Gan, W.J.~Liu, Block-transitive point-primitive automorphism groups of Steiner $3$-designs. \url{http://arxiv.org/abs/2112.00466v3}

 
 \bibitem{G2007}
N.~Gill, $PSL(3,q)$ and line-transitive linear spaces, \emph{Beitr\"{a}ge zur Algebra und Geometrie/Contributions to Algebra and Geometry} 48 (2007), 591--620. \url{https://www.emis.de/journals/BAG/vol.48/no.2/16.html}
 
 

\bibitem{CFSG}
 D. Gorenstein, R. Lyons, R. Solomon, \emph{The Classification of the Finite Simple Groups, Number 3}, Mathematical Surveys and Monographs, vol. 40, American Mathematical Society, Providence, RI, 1998. \url{https://doi.org/http://dx.doi.org/10.1090/surv/040.3}


\bibitem{H2005}
M.~Huber, The classification of flag-transitive Steiner $3$-designs, \emph{Advances in Geometry} 5 (2005), 195--221. \url{https://doi.org/10.1515/advg.2005.5.2.195}



\bibitem{H2007-4flag}
M.~Huber, The classification of flag-transitive Steiner $4$-designs, \emph{Journal of Algebraic Combinatorics} 26 (2007), 183--207. \url{https://doi.org/10.1007/s10801-006-0053-0}


\bibitem{H2009-5flag}
M.~Huber, The classification of flag-transitive Steiner $5$-designs, in: M.~Huber, \emph{Flag-transitive Steiner Designs} Birkh\"{a}user, Basel, Berlin, Boston, 2009, 93--109. \url{https://doi.org/10.1007/978-3-0346-0002-6_9}
 

\bibitem{H2010t=7} 
M.~Huber, On the existence of block-transitive combinatorial designs, \emph{Discrete Mathematics \& Theoretical Computer Science} 12 (2010), 123--132. \url{https://hal.inria.fr/hal-00990459}  

 \bibitem{H2010}
M.~Huber, On the Cameron--Praeger conjecture, \emph{Journal of Combinatorial Theory, Series A} 117 (2010), 196--203. \url{https://doi.org/10.1016/j.jcta.2009.04.004}  


\bibitem{H2010affine}
M.~Huber, Block-transitive designs in affine spaces Designs, \emph{Designs, Codes and Cryptography} (2010) 55, 235--242. \url{https://doi.org/10.1007/s10623-009-9338-3}  


\bibitem{BAtlas}
C. Jansen, K. Lux, R. Parker, R. Wilson, \emph{An Atlas of Brauer Characters}, Oxford University Press, 1995.
\url{https://doi.org/10.1112/blms/29.1.117}

\bibitem{K1972}
 W.M.~Kantor, $k$-homogeneous groups, \emph{Mathematische Zeitschrift}  124 (1972), 261--265.
 \url{https://doi.org/10.1007/BF01113919}  

 

\bibitem{K-Lie}
P.B.~Kleidman, M.W.~Liebeck, \emph{The Subgroup Structure of the Finite Classical Groups}, 
Cambridge University Press, Cambridge, 1990.
\url{https://doi.org/10.1017/CBO9780511629235}  


\bibitem{K1988-3D4q}
P.B.~Kleidman, The maximal subgroups of the Steinberg triality groups ${}^3D_4(q) $ and of their automorphism groups, \emph{Journal of Algebra} 115 (1988), 182--199.
\url{https://doi.org/10.1016/0021-8693(88)90290-6}  

\bibitem{K1988-G2q2G2q}
 P.B.~Kleidman, The maximal subgroups of the Chevalley groups $G_2(q)$ with $q$ odd, of the Ree groups ${}^2G_2 (q)$, and of their automorphism groups, \emph{Journal of Algebra} 117 (1988), 30--71.
\url{https://doi.org/10.1016/0021-8693(88)90239-6}  
 
 
 
  \bibitem{KMMM2013}
 K. Kutnar, A. Maln\v{i}c, L. Mart\'{i}nez \and D. Maru\v{s}i\v{c},  Quasi $m$-Cayley strongly regular graphs, \emph{Journal of the Korean Mathematical Society} 50 (2013), 1199-1211.
 \url{http://dx.doi.org/10.4134/JKMS.2013.50.6.1199}  
 
 \bibitem{LLY2022+}
T.~Lan, W.J.~Liu, F.-G.~Yin,  Block-transitive $3$-$(v,k,1)$ designs associated with alternating groups. \url{https://arxiv.org/abs/2208.00670}
 
 
 
  

%
%

\bibitem{CLS1992}
M.W.~Liebeck, J. Saxl, G.M.~Seitz, Subgroups of maximal rank in finite exceptional groups of Lie type, \emph{Proceedings of the London Mathematical Society} 65 (1992),  297--325. \url{https://doi.org/10.1112/plms/s3-65.2.297}
 





%
%




%
  
%
%
%
%
 
 
\bibitem{L2001}
W.J.~Liu, The Chevalley groups $G_2(2^n)$ and $2$-$(v, k, 1)$ designs, \emph{Algebra Colloquium} 8 (2001), 471--480.

\bibitem{L2003}
W.J.~Liu, Finite linear spaces admitting a two-dimensional projective linear group, \emph{Journal of Combinatorial Theory, Series A} 103 (2003), 209--222. \url{https://doi.org/10.1016/S0097-3165(03)00013-X}
 


\bibitem{L2003a}
W.J.~Liu, The Chevalley groups $G_2(q)$ with $q$ odd and $2$-$(v,k,1)$ designs, \emph{European Journal of Combinatorics}  24 (2003), 331--346. \url{https://doi.org/10.1016/S0195-6698(02)00145-2}
 




\bibitem{L2003b}
W.J.~Liu, Finite linear spaces admitting a projective group $PSU(3,q)$ with $q$ even, \emph{Linear Algebra and its Applications} 374 (2003), 291--305. \url{https://doi.org/10.1016/S0024-3795(03)00610-4}
 




\bibitem{LDG2006}
W.J.~Liu, S.J.~Dai, L.Z.~Gong, Almost simple groups with socle ${}^3D_4(q)$ act on finite linear spaces, \emph{Science in China: Series A Mathematics} 49 (2006), 1768--1776. \url{https://doi.org/10.1007/s11425-006-2040-2}
 



\bibitem{LLG2006}
W.J.~Liu, S.Z.~Li, L.Z.~Gong, Almost simple groups with socle $ Ree(q)$ acting on finite linear spaces, \emph{European Journal of Combinatorics} 27 (2006), 788--800. \url{https://doi.org/10.1016/j.ejc.2005.05.008}




\bibitem{LLM2001}
W.J.~Liu, H.L.~Li, C.G.~Ma, Suzuki groups and $2$-$(v, k, 1)$ designs, \emph{European Journal of Combinatorics}  22 (2001), 513--519. \url{https://doi.org/10.1006/eujc.2000.0413}




\bibitem{LZLF2004} 
W.J.~Liu, S.L.~Zhou, H.L.~Li, X.G.~Fang, Finite linear spaces admitting a Ree simple group, \emph{European Journal of Combinatorics} 25 (2004), 311--325. \url{https://doi.org/10.1016/j.ejc.2003.10.001}






 
 \bibitem{L1964}
H.~L\"{u}neburg, Finite  Mobius  planes  admitting  a Zassenhaus  group  as  group  of  automorphisms, \emph{Illinois Journal of Mathematics} 8 (1964), 586--592.  \url{https://doi.org/10.1215/ijm/1256059457}
  
 
%
%
%


\bibitem{MT2001}
A.~Mann, N.D.~Tuan, Block-transitive point-imprimitive $t$-designs, \emph{Geometriae Dedicata} 88 (2001), 81--90. \url{https://doi.org/10.1023/A:1013184931427}






\bibitem{RCW1975}
 D.K.~Ray-Chaudhuri, R.M.~Wilson, On $t$-designs, \emph{Osaka Journal of Mathematics} 12 (1975), 737--744. \href{https://doi.org/10.18910/7296}{https://doi.org/10.18910/7296}
 	

 
\bibitem{Suzuki1962}
M. Suzuki, On a class of doubly transitive groups, \emph{Annals of Mathematics, Second Series} 75 (1962) 105--145. \url{https://doi.org/10.2307/1970408} 
 	
 

%
 

 
 
\bibitem{GAP4}
  The GAP~Group, \emph{GAP -- Groups, Algorithms, and Programming, 
  Version 4.12.2} (2022). \url{https://www.gap-system.org}
  
  
\bibitem{Vasilev2}
A.V.~Vasilyev, Minimal permutation representations of finite simple exceptional groups of types $E_6$, $E_7$ and $E_8$, \emph{Algebra and Logic} 36 (1997), 302--310. \url{https://doi.org/10.1007/BF02671607} 

\bibitem{Z2002}
S.L.~Zhou, Block primitive $2$-$(v, k,1)$ designs admitting a Ree simple group, \emph{European Journal of Combinatorics} 23 (2002),  1085--1090. \url{https://doi.org/10.1006/eujc.2002.0616} 



\bibitem{Z2005}
S.L.~Zhou, Block primitive $2$-$(v,k,1)$ designs admitting a Ree group of characteristic two, \emph{Designs, Codes and Cryptography} 36 (2005), 159--169. \url{https://doi.org/10.1007/s10623-004-1702-8} 


\bibitem{ZLL2000}
S.L.~Zhou, H.L.~Li, W.J.~Liu, The Ree groups ${}^2G_2(q)$ and $2$-$(v,k,1)$ block designs, \emph{Discrete Mathematics} 224 (2000), 251--258. \url{https://doi.org/10.1016/S0012-365X(00)00050-9} 
 

 
 \bibitem{Zsigmondy}
K. Zsigmondy, Zur theorie der potenzreste, \emph{Monatshefte f\"{u}r Mathematik} 3 (1892), 265--284. \url{https://doi.org/10.1007/BF01692444} 
 
 
 
\end{thebibliography}
\end{document}